\documentclass[A4paper,reqno,oneside,12pt]{amsart}

\usepackage[foot]{amsaddr}
\usepackage{fullpage}
\usepackage{setspace}

\usepackage{amsmath,amssymb,amsfonts,amsthm,mathtools,mathrsfs,bbm}
\usepackage{tikz}
\usepackage{subcaption}
\usepackage{verbatim}
\usepackage{enumitem}
\usepackage{graphicx}

\usepackage{xparse} 
\usepackage{xcolor}

\usepackage[backref=page]{hyperref}
\usepackage{cleveref}
\hypersetup{%
  colorlinks,
  linkcolor={red!50!black},
  citecolor={green!50!black},
  urlcolor={blue!50!black}
}

\usepackage{lineno}
\usepackage[normalem]{ulem}

\allowdisplaybreaks

\theoremstyle{plain}
\newtheorem{theorem}{Theorem}[section]
\newtheorem*{theorem*}{Theorem}
\newtheorem{lemma}[theorem]{Lemma}

\newtheorem{corollary}[theorem]{Corollary}

\theoremstyle{definition}

\newtheorem*{remark*}{Remark}

\newcommand{\e}{\mathbbm{1}}          
\renewcommand{\d}{\,\mathrm{d}}         

\newcommand{\Z}{\mathbb{Z}}

\newcommand{\R}{\mathbb{R}}
\newcommand{\C}{\mathbb{C}}
\newcommand{\F}{\mathbb{F}}
\newcommand{\TT}{\mathbb{T}}

\newcommand{\ee}{\varepsilon}

\newcommand{\abs}[1]{\left| #1 \right|}

\title{\vspace{-0.7cm} Roth-type theorems in $K_{s,t}$-free sets}
\author{Yifan Jing}
\thanks{Department of Mathematics, the Ohio State University, Columbus, OH, USA. Email: {\tt jing.245@osu.edu}.}
\author{Cosmin Pohoata}
\thanks{Department of Mathematics, Emory University, Atlanta, GA, USA. Email: {\tt cosmin.pohoata@emory.edu}.}
\author{Max Wenqiang Xu}
\thanks{Courant Institute of Mathematical Sciences, New York, USA Email: {\tt maxxu1729@gmail.com}}

\begin{document}

\begin{abstract}
We show that for all integers $2\le s\le t$, any $K_{s,t}$-free subset of $[N]$ with size $\Omega(n^{1-1/s})$ must contain a nontrivial solution to every fixed translation-invariant linear
equation in at least five variables. This extends earlier results for Sidon sets due to
Conlon--Fox--Sudakov--Zhao and Prendiville to the full family of $K_{s,t}$-free sets. 

We also study the corresponding problem in vector spaces over finite fields.
In $\F_q^n$ we obtain stronger quantitative bounds, including polylogarithmic savings, by combining
Fourier-analytic transference with polynomial-method input from the arithmetic cycle-removal lemma of Fox--Lov\'asz--Sauermann.

\end{abstract}
\maketitle

\section{Introduction}

Fix integers $2\le s\le t$ and write $K_{s,t}$ for the complete bipartite graph with parts of sizes
$s$ and $t$. Let $\operatorname{ex}(n,K_{s,t})$ denote the maximum number of edges in an $n$-vertex graph containing no copy of $K_{s,t}$. A well-known theorem of K\H{o}v\'ari--S\'os--Tur\'an~\cite{KST54} gives
\begin{equation} \label{KST}
\operatorname{ex}(n,K_{s,t}) \ll_{s,t} n^{2-1/s}.
\end{equation}
There are several remarkable algebraic constructions showing that \eqref{KST} is optimal (up to constant factors) for various values of $s$ and $t$, in particular when $t$ is sufficiently large compared to $s$. See for example the norm graph constructions from ~\cite{KRSz96,ARSz99}, and the more recent
random-algebraic construction from \cite{BukhCrazy}. There are also many other values of $s$ and $t$ for which the asymptotic value of $\operatorname{ex}(n,K_{s,t})$ is poorly understood. For example, even for $s=t=4$, the problem is wide open. 

In this paper, we consider the following arithmetic analogue, first introduced by Erd\H{o}s and Harzheim in \cite{ErdosHarzheim1986} (see also \cite{CT18}). For integers $2\le s\le t$, we say that a set $A\subseteq [N]$ is \emph{$K_{s,t}$-free} if it contains no configuration of the form
\[
\{x_i+y_j : 1\le i\le s,\ 1\le j\le t\} \subseteq A
\]
with integers $x_1,\dots,x_s$ distinct and integers $y_1,\dots,y_t$ distinct integers. Equivalently, $A$ contains no subset of
the form $B+C$ with $|B|=s$ and $|C|=t$. A standard sum-graph reduction shows that the K\H{o}v\'ari--S\'os--Tur\'an theorem immediately implies
the trivial density bound
\[
|A| \ll_{s,t} N^{1-1/s}.
\]
When $s=t=2$, such sets are precisely \emph{Sidon sets} (also called $B_2$-sets), and the above
correspondence recovers the classical link between Sidon sets and $C_4$-free graphs. See \cite{CT18} for more details and further generalizations.

\medskip
\subsection*{Our results}
Building upon recent work of Conlon--Fox--Sudakov--Zhao~\cite{CFSZ} and Prendiville~\cite{Prendi}, we show that all dense $K_{s,t}$-free subsets of $\{1,\dots,N\}$ must possess arithmetic structure, in the sense that they must always contain a solution to any translation-invariant linear equation
in at least five variables, with all variables distinct.

\begin{theorem}\label{JPX} 
Let $k\geq 5$, and $a_1,\dots,a_k\in\mathbb{Z}\setminus\{0\}$ with $\sum_{i=1}^k a_i=0$. 
Let $A\subseteq [N]$ be $K_{s,t}$-free and assume that $A$ lacks nontrivial solutions of $a_1x_1+\cdots +a_kx_k=0$. Then
\[
|A|\ll N^{1-1/s} \exp(-O_{a_i,s,t}((\log\log N)^{\frac17})).
\]
\end{theorem}

The qualitative version of this result in the case of Sidon sets (when $s=t=2$, and with $|A| = o(N^{1/2})$) was first established by Conlon, Fox, Sudakov, and Zhao~\cite{CFSZ}, who proved it using a regularity lemma for graphs with few cycles of length four. Not too long after, Prendiville developed a Fourier-analytic transference principle to show the improved quantitative bound $|A| = O\!\left(N^{1/2}(\log\log N)^{-1}\right)$. This argument used an $L^{2}$-dense model theorem in the style of Helfgott and de Roton~\cite{HeRo} for Roth's theorem in the primes (originally established by Green~\cite{Gre05}). To prove Theorem~\ref{JPX}, we take a similar Fourier-analytic transference path, relying instead on an $L^{s}$-dense model theorem, in the spirit of Naslund's work~\cite{Naslund}.  

\medskip

\subsection*{Finite-field analogue.}
In Section~\ref{sec:finite-field} we study the same question in the vector space $\F_q^n$
($q$ a fixed odd prime power), where we establish the following theorem. 

\begin{theorem}\label{JPX2}
Let $q$ be a fixed odd prime power.  Let $k\ge 5$ and $a_1,\dots,a_k\in\F_q^\times$ satisfy $\sum_{i=1}^k a_i=0$.
Let $A\subseteq \F_q^n$ be $K_{s,t}$-free and assume that $A$ contains no nontrivial solution to
$a_1x_1+\cdots+a_kx_k=0$.  Then, by writing $N=q^n$, there exists $c=c(q,k,s,t)>0$ such that
\[
|A|\ll N^{1-1/s}(\log N)^{-c}.
\]
\end{theorem}

To obtain the polylogarithmic from Theorem~\ref{JPX2}, we take advantage of the polynomial method developments around the cap set problem~\cite{CLP17,EG}, in particular the improved bounds in the arithmetic $k$-cycle removal lemma in $\mathbb{F}_{q}^{n}$ due to Fox, Lov\'asz, and Sauermann~\cite{FLS}.


We remark that $k\ge 5$ is optimal in the sense that if we want to avoid nontrivial solutions in any given $k$ translation invariant linear equation. For example, in the $K_{2,2}$ Sidon case, both $x_1+x_2=x_3+x_4$ ($k=4$) and $x+y=2z$ ($k=3$) are excluded by the Sidon condition.

\medskip
\subsection*{Proof ideas}\label{subsec:proof-ideas}

Our arguments follow the same overarching Fourier analytic transference principle from \cite{Prendi}: we build a \emph{dense model} for a sparse set (or weight) and then transfer a \emph{dense counting statement} back to the original object.  It is conceptually helpful to split the proof into two independent components.

\medskip
\noindent{\bf (1) Dense model step (Fourier approximation).}
Given a sparse set $A$ in the natural extremal scaling (for instance $|A|\asymp N^{1-1/s}$ for
$K_{s,t}$-free sets), we introduce a renormalized
nonnegative weight $\nu$ (e.g.\ $\nu=N^{1/s}\mathbf{1}_A$) so that
$\sum \nu$ is of order $N$.  The goal is to construct a nonnegative function $f$ with
\[
\sum f=\sum \nu
\qquad\text{and}\qquad
\|\widehat f-\widehat\nu\|_\infty \ \ \text{small},
\]
while keeping some $L^s$ norm of $f$ bounded ($\sum f^s\ll N$ in the $K_{s,t}$ setting). Over the integers, the standard way to produce $f$ is to smooth by a Bohr set built from the large
spectrum; controlling the rank of this Bohr set is a key difficulty, and this step uses the $K_{s,t}$-freeness of $A$. Over $G=\F_q^n$ some of these steps will be simplified substantially, because the analogue of a Bohr set is an actual subspace.

\medskip
\noindent{\bf (2) Dense counting step (supersaturation).}
Fix a translation-invariant equation in $\mathbb{Z}[x_1, \dots, x_k]$, 
\[
a_1x_1+\cdots+a_kx_k=0,
\qquad a_i\ne0,\ \sum_{i=1}^k a_i=0,
\]
and write the associated counting functional
\[
T(h_1,\dots,h_k)
:=\sum_{a_1x_1+\cdots+a_kx_k=0}\ \prod_{i=1}^k h_i(x_i),
\qquad
T(h):=T(h,\dots,h).
\]
The dense model $f$ is built so that it behaves ``like a dense set'' from the point of view of
Fourier analysis; one then applies a dense-set input to obtain a nontrivial lower bound for
$T(f)$.  In the integer case this dense input comes from quantitative results on solutions to
translation-invariant equations in dense subsets of $[N]$ (Bloom-type bounds \cite{Bloom} and their refinements).
In the finite-field case we can instead use the arithmetic $k$-cycle removal lemma of
Fox--Lov\'asz--Sauermann \cite{FLS} (which in turn builds upon the works of Croot--Lev--Pach \cite{CLP17} and Ellenberg--Gijswijt \cite{EG}) and it yields polynomial supersaturation bounds.

\medskip
\noindent{\bf (3) Transference (telescoping + moment control).}
The final step is to compare $T(f)$ with $T(\nu_A)$.  A standard telescoping identity expands
\[
\prod_{i=1}^k f(x_i)-\prod_{i=1}^k \nu(x_i)
\]
as a sum of $k$ terms, each containing one copy of $f-\nu$ and $k-1$ copies of either $f$ or $\nu$.
To bound each resulting multilinear form we use a Fourier/H\"older estimate whose strength depends
on an available moment bound for a suitable \emph{majorant} $\omega\ge |f|+|\nu|$.
In the $K_{s,t}$-free setting, the underlying graph-freeness yields strong control on the second
moment energy $E_2$, which provides the $L^4$ Fourier bound needed for the counting lemma.
Combining the dense lower bound for $T(f)$ with the transference error bound
\[
|T(f)-T(\nu)|\ \ll\ N^{k-2}\,\|\widehat{f-\nu}\|_\infty,
\]
one concludes that $T(\nu)$ is also large.  Finally, if $A$ is assumed to have only ``trivial''
solutions, then $T(\nu)$ can be computed explicitly, which leads to an upper bound that contradicts the transferred lower bound for $N$ sufficiently large. 

\medskip

\subsection*{Notation}
For a function $f:\Z\to\C$ with finite support, we define its Fourier transform on $\TT$ by
\[
\widehat{f}(\alpha) = \sum_{n\in\Z} f(n)e(-n\alpha),
\qquad \alpha\in\TT.
\] 
The convolution of two functions on $\Z$ is denoted as
\[
f*g(x) = \sum_{y\in\Z} f(y)g(x-y).
\]
With respect to Haar probability measure on $\TT$, define the $L^p$ norm of $F:\TT\to\C$ by 
\[
\|F\|_{p} =  \Big(\int_{\TT}|F(\alpha)|^{p} \d\alpha \Big)^{\frac{1}{p}},
\qquad
\|F\|_{\infty}:= \sup_{\alpha \in \TT} |F(\alpha)|.
\]
Throughout the paper, we write $f=O(g)$ or $f\ll g$ if there exists a positive constant $C$ such that $|f(x)|\le C g(x)$ for all $x$, and $f\asymp g$ if $f=O(g)$ and $g=O(f)$. We say $f= o(g)$ if for any $\ee>0$, $|f(x)|\le \ee g(x)$ for all sufficiently large $x$.

\medskip

\subsection*{Acknowledgments} CP was supported by NSF grant DMS-2246659. MWX is supported by a Simons Junior Fellowship from the Simons Foundation.  

\section{$K_{s,t}$-free sets and the $s$-th moment energy}\label{sec:energy}

Let $f_1,\dots,f_h$ be real-valued $1$-bounded finitely supported functions on $[N]$. We define the \emph{$s$-th moment energy} $E_{s}$ by
\[
E_{s}(f_1,\dots,f_h)=\sum_{n\in \mathbb{N}}(f_1*\cdots *f_h(n))^s.
\]
When $f_1=\dots=f_h=f$, we denote it by $E_{h,s}(f)$. 

Let us first discuss the trivial bounds for the $s$-th moment energy.  Let $A\subset [N]$.
If each $f_i$ is either $\e_A$ or $\e_{-A}$, then one has the crude estimate
\[
|A|^{h} \le E_{s}(f_1,\dots,f_h)\le |A|^{sh-s+1}.
\]
Indeed, the lower bound follows from $E_s\ge E_1$ and the identity $E_1(f_1,\dots,f_h)=|A|^h$. For the upper bound, note that
\[
E_{s}(f_1,\dots,f_h)\leq \Big(\sup_n f_1*\cdots * f_h(n)\Big)^{s-1}E_{1}(f_1,\dots,f_h),
\]
and $\sup_n f_1*\cdots * f_h(n)\le |A|^{h-1}$ holds trivially, since for each fixed $n$ there are at most $|A|^{h-1}$ choices of $(a_1,\dots,a_{h-1})\in A^{h-1}$ and then $a_h$ is determined.

A good control on the higher moment energy $E_s$ naturally gives a good consequence on lower moment energy $E_{s-1}$ and $E_2$. 
\begin{lemma}\label{lem: E2}
Suppose $A\subseteq [N]$ satisfies $E_s(\e_A,\e_{-A})\ll |A|^s$ for integer  $s\ge 2$. Then
\[
E_{2}(\e_A,\e_{-A})\ll |A|^{3-\frac{1}{s-1}},
\qquad
E_{s-1}(\e_A,\e_{-A})\ll |A|^{s-\frac{s-2}{s-1}}.
\]
\end{lemma}

\begin{proof}
Let $f=\e_A *\e_{-A}$. Note that $\|f\|_1=|A|^2$. An application of H\"older's inequality implies
\[ E_2(\e_A,\e_{-A}) = 
\|f\|^2_2\leq \left\|f^{\frac{s-2}{s-1}}\right\|^2_{\frac{s-1}{s-2}} \cdot \left\|f^{\frac{s}{s-1}}\right\|^2_{s-1}\ll |A|^{3-\frac{1}{s-1}}.
\]
The bound for $E_{s-1}$ follows similarly by interpolating between $L^1$ and $L^s$.
\[
\|f\|^{s-1}_{s-1}\leq \left\|f^{\frac{s(s-2)}{s-1}}\right\|^{s-1}_{\frac{s-1}{s-2}}\cdot \left\|f^{\frac{1}{s-1}}\right\|^{s-1}_{s-1}\ll |A|^{s-\frac{s-2}{s-1}},
\]
as claimed. 
\end{proof}

We next explore how the $K_{s, t}$ freeness leads to a control on the high moment energy $E_s$. 
\begin{lemma}\label{lem: E_s}
Let $2\le s\le t$ and let $A\subseteq [N]$ be $K_{s,t}$-free. Then there is a constant $c_s\in(0, 1]$ such that
\[
E_s(\e_A,\e_{-A})\leq t|A|^s+O(|A|^{s-c_s}). 
\]
\end{lemma}

\begin{proof}
Write the representation function  
\[
r(d):=(\e_A*\e_{-A})(d)=\#\{(a,a')\in A^2:\ a-a'=d\}.
\]
Then
\begin{equation}\label{eq:Es-expand}
E_s(\e_A,\e_{-A})
=\sum_{d} r(d)^s
=\#\Bigl\{(a_1,\dots,a_s,a_1',\dots,a_s')\in A^{2s}:\ a_1-a_1'=\cdots=a_s-a_s'\Bigr\}.
\end{equation}
We split the count according to the common difference $d:=a_i-a_i'$. First, note that if $d=0$ then $a_i'=a_i$ for all $i$, so this contributes exactly $|A|^s$ to $E_s$.

Fix an ordered $s$-tuple $(a_1,\dots,a_s)\in A^s$ with pairwise distinct entries. For $d\in\Z$ define the set of admissible shifts
\[
D(a_1,\dots,a_s):=\{d\neq 0:\ a_1-d,\dots,a_s-d\in A\}.
\]
Note that for each $d\in D(a_1,\dots,a_s)$ there is exactly one corresponding
$(a_1',\dots,a_s')$, namely $a_i'=a_i-d$. Hence the number of nonzero-$d$ configurations in
\eqref{eq:Es-expand} with the given $(a_1,\dots,a_s)$ is precisely $|D(a_1,\dots,a_s)|$.

We claim that
\begin{equation}\label{eq:Ds-bound}
|D(a_1,\dots,a_s)|\le t-1.
\end{equation}
Indeed, suppose for contradiction that $|D(a_1,\dots,a_s)|\ge t$.
Choose $t$ distinct shifts $d_1,\dots,d_t\in D(a_1,\dots,a_s)$ and form the $s\times t$ grid
\[
a_{i,j}:=a_i-d_j\in A\qquad(1\le i\le s,\ 1\le j\le t).
\]
For each fixed column $j$ we have $a_i-a_{i,j}=d_j$, so in particular for every $2\le j\le t$,
\begin{equation}\label{eq:equiv-def-expanded}
a_{1,1}-a_{1,j}=a_{2,1}-a_{2,j}=\cdots=a_{s,1}-a_{s,j}.
\end{equation}
Define
\[
x_i:=a_{i,1}=a_i-d_1,\qquad
y_j:=a_{1,j}-a_{1,1}=-(d_j-d_1).
\]
Then for every $i,j$ we can rewrite \eqref{eq:equiv-def-expanded} as
\[
a_{i,j}
= a_{i,1} + (a_{1,j}-a_{1,1})
= x_i+y_j.
\]
Since the $a_i$ are pairwise distinct, the $x_i=a_{i,1}$ are also pairwise distinct.
Since the $d_j$ are distinct, the differences $d_j-d_1$ are distinct, and the $y_j$ are distinct as well.
Thus we have produced sets
\[
B:=\{x_1,\dots,x_s\},\qquad C:=\{y_1,\dots,y_t\},
\]
with $|B|=s$, $|C|=t$, and
\[
B+C=\{x_i+y_j:\ 1\le i\le s,\ 1\le j\le t\}\subseteq A,
\]
contradicting that $A$ is $K_{s,t}$-free. This proves \eqref{eq:Ds-bound}.

Consequently, for every \emph{distinct} $(a_1,\dots,a_s)\in A^s$ there are at most $t-1$ nonzero shifts $d$
(and hence at most $t-1$ nontrivial choices of $(a_1',\dots,a_s')$) contributing to \eqref{eq:Es-expand}.

Let $\mathcal{D}\subseteq A^s$ be the set of $s$-tuples with pairwise distinct entries, and let
$\mathcal{R}:=A^s\setminus\mathcal{D}$ be the set of tuples with  repeated entries.
From the discussion above and \eqref{eq:Ds-bound}, the nonzero-$d$ contribution coming from $\mathcal{D}$
is at most $(t-1)|\mathcal{D}|\le (t-1)|A|^s$.

It remains to control the contribution coming from $\mathcal{R}$. This is a lower-dimensional
degenerate set of configurations; one convenient bound is that for each fixed $d\neq 0$, the number
of $s$-tuples in $A^s$ with a repetition and with $a_i-d\in A$ for all $i$ is $O_s(r(d)^{s-1})$,
so the total contribution from $\mathcal{R}$ is
\[
\ll_s \sum_{d\neq 0} r(d)^{s-1}\leq E_{s-1}(\e_A,\e_{-A}).
\]
(Here we use that having at least one repetition means the $s$-tuple is determined by at most $s-1$
independent choices of pairs $(a,a-d)\in A^2$.)

By the estimates above and the trivial bound $ E_{s-1}(\e_A,\e_{-A})\ll_{s} |A|^{s}$, we have 
\begin{equation*}\label{eqn: Es}
E_s(\e_A,\e_{-A})\ll_{s}  (t-1)|A|^{s} + E_{s-1}(\e_A,\e_{-A})\ll_{s, t} |A|^{s}.    
\end{equation*}
 Thus 
Lemma~\ref{lem: E2} is applicable and yields
\[
E_{s-1}(\e_A,\e_{-A})\ll |A|^{\,s-\frac{s-2}{s-1}}.
\]
Putting everything together, we get
\[
E_s(\e_A,\e_{-A})
\le |A|^s \;+\; (t-1)|A|^s \;+\; O(|A|^{s-\frac{s-2}{s-1}})
\]
as required for $s>2$, where we set $c_s = \frac{s-2}{s-1}$.

For $s=2$, we simply have 
\[
E_2(\e_A,\e_{-A})=|A|^2 +\sum_{d\neq0} (t-1)r(d) \leq t|A|^2 - (t-1)|A|,
\]
hence $c_2=1$ works. 
\end{proof}
The next lemma records a control on the  representation function, given an upper bound on the high moment energy. 
\begin{lemma}\label{lem:rA-large} Let $\eta\in(0,1)$ and $2\le s\le t$.   
Suppose $A\subseteq [N]$ satisfies $E_s(\e_A,\e_{-A})\leq (t+\eta)|A|^s$.
For $a_1,\dots,a_s\in A$ define
\[
r_{A}(a_1,\dots,a_s)=\sum_{\substack{a_1-a_1'=a_2-a_2'=\cdots = a_s-a_s'\neq 0}}\e_A(a_1')\cdots \e_A(a_{s}').
\]
Then
\[
\sum_{\substack{a_1,\dots,a_s\in A\\ r_{A}(a_1,\dots,a_s)>t-1}} 1
\leq \left(1-\frac{1-\eta}{t}\right) |A|^{s}. 
\]
\end{lemma}

\begin{proof}
Observe that
\begin{align*}
  \sum_{\substack{a_1,\dots,a_s\in A\\ r_{A}(a_1,\dots,a_s)>t-1}} 1 \leq \sum_{\substack{a_1,\dots,a_s\in A\\ r_{A}(a_1,\dots,a_s)>t-1}} r_A(a_1,\dots,a_s) -(t-1)\sum_{\substack{a_1,\dots,a_s\in A\\ r_{A}(a_1,\dots,a_s)>t-1}} 1.
\end{align*}
Therefore,
\begin{align*}
t\sum_{\substack{a_1,\dots,a_s\in A\\ r_{A}(a_1,\dots,a_s)>t-1}} 1 \leq \sum_{a_1,\dots,a_s\in A} r_A(a_1,\dots,a_s)=(t-1+\eta)|A|^s,
\end{align*}
as desired.
\end{proof}

\begin{lemma}\label{lem: upper bound on A} Let $\eta\in(0,1)$  and $2\le s\le t$.
Suppose $A\subseteq [N]$ and $E_s(\e_A,\e_{-A})\leq (t+\eta)|A|^s$. Then
\[
|A|\leq 2\Bigl(\frac{t^2}{1-\eta}\Bigr)^{1/s}N^{1-1/s}.
\]
\end{lemma}

\begin{proof}
Split $A^s=\{(a_1, \dots, a_s)\}$ according to whether $r_A(a_1,\dots,a_s)>t-1$:
\begin{align*}
|A|^{s} =\sum_{\substack{a_1,\dots,a_s\in A\\ r_A(a_1,\dots,a_s)>t-1}} 1
+\sum_{\substack{a_1,\dots,a_s\in A\\ r_A(a_1,\dots,a_s)\le t-1}} 1 \le \left(1-\frac{1-\eta}{t}\right) |A|^{s}
+\sum_{\substack{a_1,\dots,a_s\in A\\ r_A(a_1,\dots,a_s)\le t-1}} 1,
\end{align*}
where we used Lemma~\ref{lem:rA-large}. Hence
\[
\frac{1-\eta}{t}\,|A|^s
\le \sum_{\substack{a_1,\dots,a_s\in A\\ r_A(a_1,\dots,a_s)\le t-1}} 1.
\]

We count the solutions to 
equation $a_1+n_1=\cdots=a_s+n_s$, where $(a_1,\dots,a_s)\in A^s$ and $(n_1,\dots,n_s)\in [2N]^s$ 
(e.g.\ take $n_i:=2N-a_i$).
On one hand, for a fixed $(a_1,\dots,a_s)$ the number of choices of $(n_1,\dots,n_s)\in [2N]^{s}$ with $a_1+n_1=\cdots=a_s+n_s$ is at least $N$ (shifting all $n_i$ by a common integer in a suitable range). On the other hand, for a fixed choice of $(n_1,\dots,n_s)\in [2N]^{s}$, we claim there are at most $t$ different choices of $(a_1, \cdots, a_s)$ such that $a_1+n_1 = \cdots = a_s+n_s$. Suppose the claim is wrong, then we can find $t$ more choices of $(a_1', \cdots, a_s')$ together with $(a_1, \cdots, a_s)$ but this gives $t$ solutions to $a_1 - a_1' = \cdots =a_s - a_s'$, and this contradicts ti $r_A(a_1, \cdots, a_s)\le t-1$.
The above double counting gives
\[
\sum_{\substack{a_1,\dots,a_s\in A\\ r_A(a_1,\dots,a_s)\le t-1}} 1
\le \frac{t(2N)^s}{N}.
\]
Combining the bounds gives
\[
|A|^s \le \frac{t^2(2N)^s}{(1-\eta)N},
\]
which proves the lemma.
\end{proof}

For each $x$, define $(x)_+=\max\{0,x\}$. The following easy lemma provides a vanishing behavior for $K_{s,t}$-free sets. It is essentially already proved in Lemma~\ref{lem: E_s}. 

\begin{lemma}\label{lem: vanishing K_st-free}
Let $A\subseteq[N]$ be $K_{s,t}$-free. Then\[
\sum_{a_1,\dots,a_s\in A} (r_A(a_1,\dots,a_s) - (t-1))_+\ll_{s,t}|A|^{s-c_s},
\]
for some constant $c_s\in(0,1)$ depending only on $s$. 
\end{lemma}
\begin{proof}
    Since $A$ is $K_{s,t}$-free, $r_A(a_1,\dots,a_s)\leq t-1$ for all tuples $(a_1,\dots,a_s)$ with distinct coordinates. This is equivalent to equation~\eqref{eq:Ds-bound} which has been proved in Lemma~\ref{lem: E_s}. Hence the positive part is supported on degenerate tuples. Summing over degeneracies and using Lemma~\ref{lem: E_s}, we obtain the desired conclusion. 
\end{proof}

\section{A dense model theorem}\label{sec:dense-model}
We first recall the following version of the large sieve inequality~\cite{Vau}.

\begin{lemma}\label{lem: sieve}
Suppose that $\delta_j>0$ ($j=1,2,\cdots, d$) and $\Gamma$ is a nonempty set of points $\gamma$ in $\R^{d}$ such that the open sets 
\[
U(\gamma) : = \{\boldsymbol{\beta}=(\beta_1,\dots,\beta_d): \|\beta_j-\gamma_j\|<\delta_j,\ 0\le \beta_j<1 \}
\]
are pairwise disjoint. Let $N_1,\dots,N_{d}$ be natural numbers and let $\mathcal{N}$ be the set of $d$-tuples $n = (n_1, \dots, n_{d})$ with $1\le n_j \le N_j$. Then the sum 
\[
S(\boldsymbol{\beta}) = \sum_{n\in \mathcal{N}}  a(n)e(n\cdot\boldsymbol{\beta}),
\]
where $a(n)$ are complex numbers, satisfies
\[
\sum_{\gamma \in \Gamma } |S(\gamma)|^{2} \ll \sum_{n \in \mathcal{N}} |a(n)|^{2} \prod_{j=1}^{d} (N_j + \delta_j^{-1}).
\]
\end{lemma}

Next we prove the following dense model theorem, which is the main result of the section. 

\begin{lemma}\label{lem: dense model kst} Let $t, s, \eta, \ee\ge 0$ be fixed.
Suppose $A\subseteq [N]$ satisfies $|A|>\delta N^{1-1/s}$ and $A$ is $K_{s,t}$-free. 
Then there is $f:[-\varepsilon N, (1+\varepsilon)N]\cap\Z\to[0,\infty)$ such that the following holds:
\begin{enumerate}[label=(\roman*)]
    \item $\sum_n f(n)= N^{\frac{1}{s}}|A| $;
    \item $\| \widehat{f}-N^{\frac{1}{s}} \widehat{\e_A}\|_\infty\ll \varepsilon N$;
    \item $\sum_n f^s(n)\leq N\Bigl(t+\exp(\varepsilon^{-O_s(1)})\cdot O_{s,t}\!\Bigl(\frac{t^2\eta}{1-\eta}\Bigr)\Bigr)$ for some $\eta\ll_{s,t}|A|^{-c_s}$.
\end{enumerate}
\end{lemma}

\begin{proof}
As usual we define the spectrum set
\[
\mathrm{Spec}(A,\varepsilon) = \{\alpha\in \TT: |\widehat{\e_A}(\alpha)|\geq \varepsilon |A|\}. 
\]
Consider the Bohr set
\[
B=\{n\in [-\varepsilon N,\varepsilon N]\cap\Z: \|n\alpha\|_\TT<\varepsilon \text{ for every } \alpha\in \mathrm{Spec}(A,\varepsilon)\}. 
\]
We use $\mu_B$ to denote the normalized characteristic function on $B$. Define
\[
f=N^{\frac1s}\e_A*\mu_B.
\]
Clearly $f$ satisfies (i). Now we consider $\widehat{f}-N^{\frac{1}{s}}\widehat{\e_A}$. Note that
\[
   | \widehat{f}-N^{\frac{1}{s}}\widehat{\e_A} |= N^{\frac1s}|\widehat{\e_A}|\, \bigl|1-\widehat{\mu_B}\bigr|. 
\]
For $\alpha\notin \mathrm{Spec}(A,\varepsilon)$, by definition and  Lemma~\ref{lem: upper bound on A} we get
\[
| \widehat{f}(\alpha)-N^{\frac{1}{s}}\widehat{\e_A} (\alpha)|\leq 2N^{\frac1s}\varepsilon|A|\ll \varepsilon N.
\]
 For $\alpha\in \mathrm{Spec}(A,\varepsilon)$, we have $\widehat{\mu_B}(\alpha)=1+O(\varepsilon)$, and again
\[
| \widehat{f}(\alpha)-N^{\frac{1}{s}}\widehat{\e_A} (\alpha)|\ll \varepsilon N^{\frac1s}|A|\ll \varepsilon N.
\]
This proves (ii).

Finally, we write 
\[
\sum_n f^s(n) = N |B|^{-s} \sum_n (\e_A*\e_B(n))^s.
\]
Noticing that 
\begin{align*}
    \sum_n (\e_A*\e_B(n))^s
    &=\sum_{a_1+b_1=\cdots =a_s+b_s} \e_A(a_1)\cdots\e_A(a_s)\e_B(b_1)\cdots\e_B(b_s)\\
    &=\sum_{\substack{a_1+b_1=\cdots =a_s+b_s\\ r_A(a_1,\dots,a_s)\leq t-1}} \e_A(a_1)\cdots\e_A(a_s)\e_B(b_1)\cdots\e_B(b_s) \\
    &\quad + \sum_{\substack{a_1+b_1=\cdots =a_s+b_s\\ r_A(a_1,\dots,a_s)> t-1}} \e_A(a_1)\cdots\e_A(a_s)\e_B(b_1)\cdots\e_B(b_s)\\
    &\leq t|B|^s+ \eta |A|^{s}|B|,
\end{align*}
for some $\eta\ll_{s,t}|A|^{-c_s}$, 
where in the last step we used Lemma~\ref{lem: vanishing K_st-free}.

It remains to give a lower bound on $|B|$. We use a standard probabilistic argument to achieve this. 
Let $\mathcal{P}=\{p_1,\dots, p_m\}$ be a maximal $1/N$-separated subset of $\mathrm{Spec}(A,\varepsilon)$. By maximality,
\[
B':=\{ n\in [-\varepsilon N/2,\varepsilon N/2]\cap\Z: \| p_in\|_\TT\leq \varepsilon/2 \text{ for every } 1\leq i\leq m\}\subseteq B.
\]
For each $p_i\in\mathcal{P}$, choose $\theta_i\in\TT$ uniformly at random. Then for every $x\in [-\varepsilon N/2,\varepsilon N/2]\cap\Z$,
\[
\mathbb{P}(\|p_ix-\theta_i\|_\TT<\varepsilon/4\text{ for every } 1\leq i\leq m)=\left(\frac{\varepsilon}{2}\right)^m.
\]
Define
\[
B'':=\{x\in [-\varepsilon N/2,\varepsilon N/2]\cap\Z:\|p_ix-\theta_i\|_\TT<\varepsilon/4\text{ for every } 1\leq i\leq m\}. 
\]
Hence $\mathbb{E}\abs{B''}\geq \left(\frac{\varepsilon}{2}\right)^{m+1} N.$
By pigeonhole principle, there is a choice of $\{\theta_i\}_{1\le i \le m}$ such that $\abs{B''}\geq (\varepsilon/2)^{m+1}N$. Note that for $x,x'\in B''$, their difference satisfies $x-x'\in B'$. This gives
\[
 \abs{B'}\geq \left(\frac{\varepsilon}{2}\right)^{m+1} N.
\]

Finally, note that $|\widehat{\e_A}|^2=\widehat{\e_A*\e_{-A}}$. By applying Lemma~\ref{lem: sieve} (with $d=1$) and Lemma~\ref{lem: E2}, we have
\[
    \varepsilon^4 m |A|^4\leq \sum_{i=1}^m |\widehat{\e_A}(p_i)|^4\ll N E_2(\e_A,\e_{-A})\ll N|A|^{3-\frac{1}{s-1}}.
\]
As $|A|>\delta N^{1-1/s}$, we have $m\ll_{s} \delta^{-2}\varepsilon^{-4}$. By choosing $\varepsilon\in(0,\delta)$, it gives  
\[
|B|\geq |B'|\geq \exp(-\varepsilon^{-O_s(1)})N. 
\]
Substituting this into the previous inequality gives (iii).
\end{proof}

\section{Counting Lemmas}\label{sec:counting}

We use a variant of the following counting result by Ko\'sciuszko~\cite{K25}, which quantitatively improves a result of Bloom~\cite{Bloom}.

\begin{lemma}[Ko\'sciuszko~\cite{K25}]\label{Kosciuszko counting}
Let $a_1,\dots,a_k \in \Z\setminus \{0\}$ with $k\ge 5$ and $a_1 +\cdots +a_k = 0$. Then for any $A \subseteq [N]$ of size $\delta N$, we have the lower bound
\[
\sum_{a_1x_1+\cdots +a_kx_k = 0}  \e_A(x_1)\cdots \e_A(x_k)
\ge \exp(-O_{a_i} ( \log^7(1/\delta)))N^{k-1}.
\]
\end{lemma}

Our application requires a $L^p$-weighted version of the above result, which can be derived by a level set argument. 

\begin{lemma}\label{L^p Kosciuszko counting}
    Let $a_1, a_2, \cdots, a_k \in \mathbb{Z}\backslash \{0\}$ with $k\ge 5$ and $a_1 +\cdots a_k = 0$, $p\geq2$. Let $f:I\to[0,\infty)$ be a function defined in an interval $I\subseteq\mathbb{Z}$ of size $|I|=N$. Suppose $\sum_n f(n)\geq \delta N$ and $\sum_n f(n)^p\leq N$. 
    Then we have the lower bound
\[ 
\sum_{a_1x_1+\cdots +a_kx_k = 0} f(x_1) f(x_2)\cdots f(x_k) \ge \exp(-O_{a_i,p} ( \log^7(1/\delta)))N^{k-1}.
\] 
\end{lemma}
\begin{proof}
  Let $Z=f(n)$ be a random variable on $I$, which is uniformly and independent distributed for $n\in I$. Then the conditions are the same as saying $\mathbb E(Z)\geq\delta$ and $\mathbb E (Z^p)\leq 1$. By the Paley--Zygmund level set inequality \cite{PZineq},
  \[
\mathbb{P}\Big(Z\geq \frac12\mathbb E(Z)\Big)\geq \Big(\frac{1}{2}\Big)^{\frac{p}{p-1}}\frac{(\mathbb E Z)^{\frac{p}{p-1}}}{\mathbb (\mathbb E(Z^p))^{\frac{1}{p-1}}}\geq\Big(\frac{\delta}{2}\Big)^{\frac{p}{p-1}}. 
  \]
  This implies that, by defining $A=\{n: f(n)\geq \delta/2\}$, we have $|A|\geq (\frac{\delta}{2})^{\frac{p}{p-1}}N$. Now by Lemma~\ref{Kosciuszko counting}, for $k\geq 5$,
  \[ 
\sum_{a_1x_1+\cdots +a_kx_k = 0} 1_A(x_1) 1_A(x_2)\cdots 1_A(x_k) \ge \exp(-O_{a_i,p} ( \log^7(1/\delta)))N^{k-1}.
\]
As $f(n)\geq \delta/2$ on $A$, we then have
 \begin{align*}
\sum_{a_1x_1+\cdots +a_kx_k = 0} f(x_1) f(x_2)\cdots f(x_k) &\geq (\delta/2)^k\exp(-O_{a_i,p} ( \log^7(1/\delta^2)))N^{k-1}\\
&\geq \exp(-O_{a_i,p} ( \log^7(1/\delta)))N^{k-1},
\end{align*}
 as the $(\delta/2)^k$ factor will be absorbed by the $\exp(-O_{a_i,p} ( \log^7(1/\delta)))$ term. 
\end{proof}



\begin{lemma}[Counting lemma for $E_2$]\label{lem: counting}
Let $a_1,\dots,a_k\in\Z \setminus \{0\}$ with $k\geq 5$ and $a_1 +\cdots + a_k = 0$. Let $\nu : I \to [0, \infty)$ be a function defined on an interval $I \subseteq \Z$ of length $N$. Suppose $\sum_n \nu(n)\ll N$ and
\[
E_2(\nu,\nu)\ll N^{3}. 
\]
Then for every $|f_i|\le \nu$ we have
\[
\left| \sum_{a_1x_1+\cdots +a_k x_k = 0} \prod_{j=1}^k f_j(x_j)\right|\ll N^{k-2}\min_i \|\widehat{f_i}\|_\infty. 
\]
\end{lemma}

\begin{proof}
By Fourier inversion,
one has
\[
 \sum_{a_1x_1+\cdots +a_k x_k = 0} \prod_{j=1}^k f_j(x_j)
=\int_\TT \prod_{j=1}^k \widehat{f_j}(a_j\alpha)\d\alpha.
\]
Fix any $i$ between $1$ and $k$. By H\"older inequality, it holds that
\[
\left| \int_\TT \prod_{j=1}^k \widehat{f_j}(a_j\alpha)\d\alpha\right|
\leq \|\widehat{f_i}\|_\infty \prod_{j\neq i}\left(\int_\TT |\widehat{f_j}(\alpha)|^{k-1}\d\alpha\right)^{\frac{1}{k-1}}. 
\]
For the inner integral, we have,
\[
\int_\TT |\widehat{f_j}(\alpha)|^{k-1}\d\alpha
\leq \|\widehat{f_j}\|_\infty^{k-5}\int_{\TT}|\widehat{f_j}(\alpha)|^{4}\d\alpha
\ll N^{k-5}E_2(f_j,f_j)\ll N^{k-2},
\]
where we used $\|\widehat{f_j}\|_\infty\le \sum_n |f_j(n)|\le \sum_n \nu(n)\ll N$ and 
\[
\int_\TT |\widehat{f_j}(\alpha)|^4d\alpha = E_2(f_j,f_j)\le E_2(\nu,\nu)\ll N^3.
\]
This gives the desired bound.
\end{proof}

\section{Proof of Theorem~\ref{JPX}}\label{sec:proof-integers}
In this section, we prove our main result in the integer setting. 

\begin{proof}[Proof of Theorem~\ref{JPX}]
Write $|A|=\delta N^{1-1/s}$.
Let $\varepsilon\in(0,1/10)$ be a parameter to be chosen later. By Lemma~\ref{lem: E_s} we have $E_s(\e_A,\e_{-A})\le (t+o(1))|A|^s$, so Lemma~\ref{lem: dense model kst} yields a function
\[
f:I:=[-\varepsilon N,(1+\varepsilon)N]\cap\Z\to[0,\infty)
\]
such that
\begin{enumerate}[label=(\roman*)]
\item $\sum_{n\in I} f(n)=N^{1/s}|A|=\delta N$,
\item $\|\widehat f-N^{1/s}\widehat{\e_A}\|_\infty\ll \varepsilon N$,
\item $\sum_{n\in I} f(n)^s\ll_{s,t} N$,
\end{enumerate}
whenever 
\begin{equation}\label{eq: bound for epsilon}
      \exp(\varepsilon^{-O_s(1)})\ll N^{c_s}. 
\end{equation}

Define the multilinear counting form
\[
T(h_1,\dots,h_k):=\sum_{a_1x_1+\cdots +a_kx_k = 0}\prod_{j=1}^k h_j(x_j),
\qquad
T(h):=T(h,\dots,h).
\]
Now apply Lemma~\ref{L^p Kosciuszko counting} to $f':=f/C^{1/s}$, where $C$ is a large absolute constant (depending only on $s,t$) coming from the hidden constant from (iii) above, so that $\sum_I f'^s\le N$ and $\sum_I f'\ge \delta N/C^{1/s}$. This yields
\begin{equation}\label{eq:Tf-lower}
T(f)\ge \exp\bigl(-O_{a_i,s,t}(\log^7(1/\delta))\bigr)\,N^{k-1}.
\end{equation}

Next, we transfer the counting result to $A$.
Set $F:=N^{1/s}\e_A$, $g:=f-F$, and $\nu:=f+F$.
Then $\|\widehat g\|_\infty\ll \varepsilon N$ by (ii) above.
By Lemma~\ref{lem: upper bound on A} and Lemma~\ref{lem: E2}, as well as H\"older inequality, we have
\begin{align*}
    E_2(\nu)^{\frac{1}{4}}&=\| \widehat{\nu}\|_{4}\ll \| \widehat{f}\|_{4}+N^{\frac{1}{s}}\| \widehat{\e_A}\|_{4}\\
    &\leq N^{\frac{1}{4}}\| f\|_2^{\frac{1}{2}}\|\widehat{f}\|_2^{\frac{1}{2}}+ N^{\frac{1}{s}}E_2(\e_A,\e_{-A})^{\frac{1}{4}}\ll N^{\frac{3}{4}}. 
\end{align*}
Notice that
$$\sum_n \nu(n) = \widehat{f}(0)+N^{1/s}\widehat{\e_A}(0) \leq 2\widehat{f}(0) + \varepsilon N \leq 2 (2N)^{1/2}\| f\|_2+\varepsilon N\ll N.$$ We use 
the telescoping identity to get
\[
\prod_{j=1}^k f(x_j)-\prod_{j=1}^k F(x_j)
=\sum_{i=1}^k g(x_i)\Bigl(\prod_{j<i} f(x_j)\Bigr)\Bigl(\prod_{j>i} F(x_j)\Bigr). 
\]
Note that $f, F \ll \nu$ and an application of Lemma~\ref{lem: counting} (viewing functions $g, F$ as $f_i$ in the lemma) gives 
\begin{equation}\label{eq:transfer}
|T(f)-T(F)|\ll_k N^{k-2}\|\widehat g\|_\infty \ll_k \varepsilon N^{k-1}.
\end{equation}
Since $A$ has no nontrivial solutions, the only solutions counted by $T(F)$ are diagonal, 
\[
T(F)=N^{k/s}\sum_{a_1x_1+\cdots +a_kx_k=0}\e_A(x_1)\cdots\e_A(x_k)
=N^{k/s}|A|=\delta\,N^{1+\frac{k-1}{s}}.
\]
Since $k\ge 5$ and $s\ge 2$, we have $1+\frac{k-1}{s}<k-1$, and $T(F)=o(N^{k-1})$.
Combining this with~\eqref{eq:transfer} gives, for $N$ large, one has
\begin{equation}\label{eq:Tf-upper}
T(f)\ll_k \varepsilon N^{k-1}.
\end{equation}
Comparing~\eqref{eq:Tf-lower} and~\eqref{eq:Tf-upper} yields
\[
\exp\bigl(-O_{a_i,s,t}(\log^7(1/\delta))\bigr)\ll_k \varepsilon.
\]
Combining with \eqref{eq: bound for epsilon}, we have
\[
\delta\ll \exp(-c(\log\log N)^{1/7}),
\]
for some $c=c(a_i,s,t) > 0$. Recalling that $|A|=\delta N^{1-1/s}$ completes the proof.
\end{proof}

\section{Finite field analogues for $K_{s,t}$-free sets}\label{sec:finite-field}

In this section we prove Theorem~\ref{JPX2} using the polynomial arithmetic $k$-cycle removal lemma
of Fox--Lov\'asz--Sauermann~\cite{FLS}. Our argument here follows the same transference scheme as in the
integer case, with the key difference that Ko\'sciuszko/Bloom type quantitative Roth inputs are replaced by a
polynomial counting lemma coming from~\cite{FLS} (whose proof ultimately relies on the
polynomial/slice-rank method of Croot--Lev--Pach~\cite{CLP17} and Ellenberg--Gijswijt~\cite{EG}).

Throughout this section let $q=p^r$ be a fixed odd prime power, let
\[
G:=\F_q^n,\qquad N:=|G|=q^n,
\]
and fix coefficients $a_1,\dots,a_k\in\F_q^\times$ with
\[
a_1+\cdots+a_k=0.
\]
As before we say a solution $(x_1,\dots,x_k)\in A^k$ to
\begin{equation}\label{eq:FF-eqn}
a_1x_1+\cdots+a_kx_k=0
\end{equation}
is \emph{trivial} if $x_1=\cdots=x_k$, and \emph{nontrivial} otherwise.

We identify the dual group $\widehat G$ with $G$ in the usual way. 
Fix once and for all a nontrivial additive character
\[
e_q:\F_q\to\C^\times,
\]
for instance
\[
e_q(u):=\exp\!\Big(\frac{2\pi i}{p}\,\mathrm{Tr}_{\F_q/\F_p}(u)\Big),
\qquad p=\mathrm{char}(\F_q).
\]
For each $\xi\in G=\F_q^n$ we define the corresponding additive character
\[
\chi_\xi:G\to\C^\times,
\qquad
\chi_\xi(x):=e_q(\langle x,\xi\rangle),
\]
where $\langle x,\xi\rangle$ denotes the standard dot product on $\F_q^n$.
Every additive character of the group $(G,+)$ arises uniquely in this way, so this gives a canonical
identification $\widehat G\simeq G$.

With this notation, for a subspace $V\le G$ we define its \emph{annihilator}
\[
V^\perp
:=\{x\in G:\ \chi_\xi(x)=1\ \text{for all }\xi\in V\}
=\{x\in G:\ \langle x,\xi\rangle=0\ \text{for all }\xi\in V\}.
\]
The annihilator $V^\perp$ is a subspace of $G$, and $(V^\perp)^\perp=V$.

For $f:G\to\C$ define
\[
\widehat f(\xi):=\sum_{x\in G} f(x)\,e_q(-\langle x,\xi\rangle)=\sum_{x\in G} f(x)\,\overline{\chi_\xi(x)}.
\]
Then Parseval identity reads
\[
\sum_{x\in G} f(x)\overline{g(x)}=\frac{1}{N}\sum_{\xi\in \widehat G}\widehat f(\xi)\overline{\widehat g(\xi)},
\]
and convolution is $(f*g)(x)=\sum_{y\in G} f(y)g(x-y)$.

All the energy lemmas from the integer setting (in particular Lemmas~\ref{lem: E2} and~\ref{lem: E_s})
extend verbatim to $G$ as they only use abelian group structure. Likewise, the counting lemma
Lemma~\ref{lem: counting} extends verbatim to $G$ once one replaces integrals over $\TT$ by averages over
the finite dual group $\widehat G$ as this uses only Parseval identity and H\"older inequality. Given this, and with abuse of notation, we may later quote lemmas proved in the integer setting but apply them in the finite field setting. 

We next discuss the analogue of the dense model Lemma \ref{lem: dense model kst} in $\mathbb{F}_{q}^{n}$.
Over $G=\F_q^n$ the dense model step simplifies substantially, because the analogue of a Bohr set is an
\emph{actual subspace}.  Indeed, for any
subspace $H\le G$ the uniform measure $\mu_H:=|H|^{-1}\e_H$ has Fourier transform
\[
\widehat{\mu_H}(\xi)=
\begin{cases}
1,& \xi\in H^\perp,\\
0,& \xi\notin H^\perp,
\end{cases}
\]
Consequently, for any function $g:G\to\C$ we have in Fourier space
\[
\widehat{g*\mu_H}(\xi)=\widehat g(\xi)\,\widehat{\mu_H}(\xi)
=\widehat g(\xi)\,\e_{H^\perp}(\xi),
\]
so convolution with $\mu_H$ keeps precisely the Fourier coefficients of $g$ on $H^\perp$ and kills all Fourier coefficients outside $H^\perp$, i.e. convolution with $\mu_H$ is an \emph{exact Fourier projector} onto the subspace $H^\perp\le \widehat G$.
Thus, by choosing $H$ to annihilate the large spectrum of $\e_A$, we can hope to smooth $\e_A$ while
preserving its large Fourier coefficients and simultaneously keeping good $L^s$ control. This is the content of the following lemma. 

\begin{lemma}[Dense model in $\F_q^n$]\label{lem:dense-model-ff}
Suppose $A\subseteq G$ satisfies $|A|>\delta N^{1-1/s}$ and $A$ is $K_{s,t}$-free. 
Then for every $0<\varepsilon<1$ there exists $f:G\to[0,\infty)$ such that
\begin{enumerate}[label=(\roman*)]
\item $\sum_{x\in G} f(x)=N^{1/s}|A|$;
\item $\|\widehat f-N^{1/s}\widehat{\e_A}\|_\infty\ll \varepsilon N$;
\item $\displaystyle \sum_{x\in G} f(x)^s\le N\Big(t+\exp(\varepsilon^{-O_s(1)})\cdot O_{s,t}(\eta)\Big)$ for some $\eta\ll_{s,t}|A|^{-c_s}$.
\end{enumerate}
\end{lemma}

\begin{proof}
Define the large spectrum
\[
\operatorname{Spec}(A,\varepsilon):=\{\xi\in G:\ |\widehat{\e_A}(\xi)|\ge \varepsilon |A|\}.
\]
Let $V:=\mathrm{Span}(\operatorname{Spec}(A,\varepsilon))\le G$ be the $\F_q$-linear span of $\operatorname{Spec}(A,\varepsilon)$, and set
\[
H:=V^\perp=\{x\in G:\ \chi_\xi(x)=1 \text{ for all }\xi\in V\}.
\]
Then $H$ is a subspace of $G$ and $\widehat{\mu_H}=\e_V$ (i.e.\ $\widehat{\mu_H}(\xi)=1$ if
$\xi\in V$ and $0$ otherwise).

We next claim that $\dim(V)\ll_{s,t,\delta}\varepsilon^{-4}$. To see this, let $\Lambda\subseteq \operatorname{Spec}(A,\varepsilon)$ be any set of linearly independent vectors spanning
$V$, so $\dim(V)=|\Lambda|$.  Then
\[
|\Lambda|\,\varepsilon^4|A|^4\ \le\ \sum_{\xi\in\Lambda}|\widehat{\e_A}(\xi)|^4
\ \le\ \sum_{\xi\in G}|\widehat{\e_A}(\xi)|^4
\ =\ N\cdot E_2(\e_A,\e_{-A}).
\]
By Lemma~\ref{lem: E2} and the hypothesis
$E_s(\e_A,\e_{-A})\ll_{s,t}|A|^s$, we have
$E_2(\e_A,\e_{-A})\ll_s |A|^{3-\frac{1}{s-1}}$, so
\[
\dim(V) = |\Lambda| \ \ll_s\ \frac{N}{\varepsilon^4|A|^{1+\frac{1}{s-1}}}.
\]
Using $|A|>\delta N^{1-1/s}$ and the identity
$(1-\tfrac1s)(1+\tfrac{1}{s-1})=1$, we obtain $\dim(V)\ll_{s,\delta}\varepsilon^{-4}$.
In particular,
\begin{equation}\label{eqn: H}
|H|=q^{n-\dim(V)}\ge q^{-O_{s,\delta}(\varepsilon^{-4})}N=\exp(-\varepsilon^{-O_s(1)})N.    
\end{equation}

Let $\mu_H:=|H|^{-1}\e_H$ and define
\[
f:=N^{1/s}\,\e_A*\mu_H.
\]
We now check that $f$ satisfies properties (i)-(iii). Clearly, $f\ge0$ and (i) is immediate since $\sum_x \mu_H(x)=1$. For (ii), using $\widehat{\mu_H}=\e_V$ we have for every $\xi\in G$,
\[
\widehat f(\xi)=N^{1/s}\widehat{\e_A}(\xi)\widehat{\mu_H}(\xi)
=
\begin{cases}
N^{1/s}\widehat{\e_A}(\xi),& \xi\in V,\\
0,& \xi\notin V.
\end{cases}
\]
Hence $\widehat f(\xi)-N^{1/s}\widehat{\e_A}(\xi)=0$ for $\xi\in V$, while for $\xi\notin V$ we have
$\xi\notin \operatorname{Spec}(A,\varepsilon)$ and therefore $|\widehat{\e_A}(\xi)|<\varepsilon|A|$.  Thus
\[
\|\widehat f-N^{1/s}\widehat{\e_A}\|_\infty
\le N^{1/s}\varepsilon|A|
\ll_{s,t}\varepsilon N,
\]
where in the last step we used a finite field version of  Lemma~\ref{lem: upper bound on A} to bound $|A|\ll_{s,t}N^{1-1/s}$.

To prove the $L^s$ bound (iii), we start by writing $\e_H$ for the (unnormalized) indicator of $H$ and by noting that
$f^s=N(\e_A*\mu_H)^s$.  Therefore
\begin{equation}\label{eqn: fs}
\sum_{x\in G} f(x)^s
=
N\sum_{x\in G} (\e_A*\mu_H(x))^s
=
N|H|^{-s}\sum_{x\in G} (\e_A*\e_H(x))^s.  
\end{equation}
For convenience, set
\[
S:=\sum_{x\in G} (\e_A*\e_H(x))^s,
\]
and rewrite the quantity as follows:
\[
S=\sum_{h_1,\dots,h_s\in H}\ \sum_{x\in G}\ \prod_{i=1}^s \e_A(x-h_i)
=\sum_{h_1,\dots,h_s\in H} |S_{\mathbf{h}}|,
\]
where $S_{\mathbf{h}}:=\{x\in G:\ x-h_i\in A\,\ \text{for all}\ 1\le i \le s\}$ and $\mathbf{h}=(h_1,\dots,h_s)$.

For an ordered $s$-tuple $\mathbf{a}=(a_1,\dots,a_s)\in A^s$ define
\[
r_A(\mathbf{a})
:=
\bigl|\{d\in G\setminus\{0\}:\ a_i-d\in A \text{ for all }1\le i\le s\}\bigr|.
\]
As $A$ is $K_{s,t}$-free, Lemma~\ref{lem: vanishing K_st-free} implies that there is $\eta = O(|A|^{-c_s})$ such that 
\begin{equation}\label{eq:excess_rA}
\sum_{\mathbf{a}\in A^s}\bigl(r_A(\mathbf{a})-(t-1)\bigr)_+ \ \le\ \eta |A|^s,
\end{equation}
 where $(x)_+ := \max\{x,0\}$ for a real number $x$. 

Next, fix $\mathbf{h}$ and suppose $S_{\mathbf{h}}\neq\emptyset$.  For any $x\in S_{\mathbf{h}}$ the
associated tuple $\mathbf{a}(x):=(x-h_1,\dots,x-h_s)\in A^s$ satisfies
\begin{equation}\label{eq:rA_equals_size}
r_A(\mathbf{a}(x))=|S_{\mathbf{h}}|-1.
\end{equation}
Indeed, if $x'\in S_{\mathbf{h}}$ then $\mathbf{a}(x')=\mathbf{a}(x)-(x-x')$ coordinatewise, giving a shift
counted by $r_A$; conversely, if $\mathbf{a}(x)-d\in A^s$ then $x-d\in S_{\mathbf{h}}$.
In particular, $|S_{\mathbf{h}}|>t$ if and only if $r_A(\mathbf{a}(x))>t-1$.

Notice that $|S_{\mathbf{h}}|\ \le\ t+\bigl(|S_{\mathbf{h}}|-t\bigr)_+$ always holds and it implies that
\begin{equation}\label{eqn: S}
   S\ \le\ t|H|^s+\sum_{\mathbf{h}\in H^s}\bigl(|S_{\mathbf{h}}|-t\bigr)_+. 
\end{equation}
If $|S_{\mathbf{h}}|>t$, then $|S_{\mathbf{h}}|\bigl(|S_{\mathbf{h}}|-t\bigr)_+\ge
t\bigl(|S_{\mathbf{h}}|-t\bigr)_+$, and further
\begin{equation}\label{eq:excess_linear}
\sum_{\mathbf{h}\in H^s}\bigl(|S_{\mathbf{h}}|-t\bigr)_+
\ \le\ \frac1t\sum_{\mathbf{h}\in H^s}|S_{\mathbf{h}}|\bigl(|S_{\mathbf{h}}|-t\bigr)_+.
\end{equation}
Using \eqref{eq:rA_equals_size}, the right hand side of \eqref{eq:excess_linear} equals
\[
\frac1t\sum_{\mathbf{h}\in H^s}\ \sum_{x\in S_{\mathbf{h}}}\bigl(r_A(\mathbf{a}(x))-(t-1)\bigr)_+.
\]
Group this sum by $\mathbf{a}\in A^s$.  For a fixed $\mathbf{a}=(a_1,\dots,a_s)$, the number of pairs
$(\mathbf{h},x)$ with $x-h_i=a_i$ for all $i$ is at most $|H|$: indeed, choosing any $h_1\in H$ determines
$x=a_1+h_1$ and then $h_i=x-a_i=h_1+(a_1-a_i)$. Therefore
\[
\sum_{\mathbf{h}\in H^s}|S_{\mathbf{h}}|\bigl(|S_{\mathbf{h}}|-t\bigr)_+
\ \le\ |H|\sum_{\mathbf{a}\in A^s}\bigl(r_A(\mathbf{a})-(t-1)\bigr)_+
\ \le\ \eta |A|^s |H|
\]
where in the last inequality we used
\eqref{eq:excess_rA}.  Combining this with \eqref{eq:excess_linear} and \eqref{eqn: S} gives
\[
S\ \le\ t|H|^s+\frac{\eta}{t}|A|^s|H|.
\]
Invoking \eqref{eqn: fs} gives
\[
\sum_{x\in G} f(x)^s
\le
N\left(t+\frac{\eta}{t}\cdot\frac{|A|^s}{|H|^{s-1}}\right).
\]
Finally, using $|A|\ll_{s,t}N^{1-1/s}$ from Lemma~\ref{lem: upper bound on A} and
$|H|\ge \exp(-\varepsilon^{-O_s(1)})N$ from \eqref{eqn: H} , we conclude that
$
\frac{|A|^s}{|H|^{s-1}}
\ \ll_{s,t}\ \exp(\varepsilon^{-O_s(1)})$,
and thereby (iii) follows.
\end{proof}

\bigskip
We next introduce the key ingredient in this section, which is an arithmetic $k$-cycle removal lemma and a corresponding polynomial counting lemma. 

\begin{theorem}[Fox--Lov\'asz--Sauermann~\cite{FLS}]\label{thm: FLS}
Let $p$ be a fixed prime and $k\ge 3$. There exists a constant $C_{p,k}>0$ such that the following holds.
Writing $N=p^m$, for every $0<\epsilon<1$ and $\delta=\epsilon^{C_{p,k}}$, and for any sets
$X_1,\dots,X_k\subseteq \F_p^m$, at least one of the following holds:
\begin{enumerate}[label=(\alph*)]
\item the number of $k$-tuples $(x_1,\dots,x_k)\in X_1\times\cdots\times X_k$ satisfying
$x_1+\cdots+x_k=0$ is at least $\delta N^{k-1}$;
\item one can delete fewer than $\epsilon N$ elements from each $X_i$ so that afterwards no such $k$-tuples remain.
\end{enumerate}
\end{theorem}

The key feature for us is the polynomial dependence $\delta=\epsilon^{C_{p,k}}$ and 
\cite{FLS} gives an explicit exponent $C_{p,k}$ which is built from the sharp
arithmetic triangle bounds.

We apply Theorem~\ref{thm: FLS} to $G=\F_q^n$ by identifying $(\F_q^n,+)$ with $\F_p^{rn}$, where $p=\mathrm{char}(\F_q)$ and $q=p^r$.
In particular, the same exponent $C_{p,k}$ works, and we denote it by $C_{q,k}$.

\begin{corollary}\label{cor:FF-supersat}
Let $A_0\subseteq G=\F_q^n$ have density $\rho:=|A_0|/N$. Then there exists a constant $C_{q,k}>0$ such that
\[
\#\{(x_1,\dots,x_k)\in A_0^k:\ a_1x_1+\cdots+a_kx_k=0\}
\ \ge\
\Big(\frac{\rho}{2k}\Big)^{C_{q,k}}\,N^{k-1}.
\]
\end{corollary}

\begin{proof}
Set $X_i:=a_iA_0\subseteq G$. Since $a_i\in\F_q^\times$, the map $x\mapsto a_ix$ is a bijection on $G$, so
$|X_i|=|A_0|=\rho N$ for all $i$.

For every $x\in A_0$, the $k$-tuple $(a_1x,\dots,a_kx)\in X_1\times\cdots\times X_k$ satisfies
\[
(a_1x)+\cdots+(a_kx)=(a_1+\cdots+a_k)x=0,
\]
so it is a $k$-cycle. Moreover, for each fixed $i$, the elements $a_ix$ are all distinct as $x$ varies,
so these $\rho N$ cycles are disjoint within each coordinate class $X_i$.

Apply Theorem~\ref{thm: FLS} to the sets $X_1,\dots,X_k$ with parameter $\epsilon':=\rho/(2k)$.
If conclusion (b) held, then deleting fewer than $\epsilon'N$ elements from each $X_i$ would delete fewer than
$k\epsilon'N=\rho N/2$ elements in total across $X_1\cup\cdots\cup X_k$.
But each deleted element in $X_i$ can destroy at most one of the disjoint cycles $(a_1x,\dots,a_kx)$,
so destroying all $\rho N$ cycles requires at least $\rho N$ deletions in total. This contradiction rules out (b), so (a) must hold.
Thus the number of cycles is at least $\delta N^{k-1}$ with $\delta=(\epsilon')^{C_{q,k}}=(\rho/2k)^{C_{q,k}}$.

Finally, the bijection $(x_1,\dots,x_k)\mapsto(a_1x_1,\dots,a_kx_k)$ identifies solutions of \eqref{eq:FF-eqn}
in $A_0^k$ with $k$-cycles in $X_1\times\cdots\times X_k$, so the same lower bound holds.
\end{proof}

The following lemma is a direct analogous of Lemma~\ref{L^p Kosciuszko counting}.

\begin{lemma}\label{lem:FF-Roth-count}
Let $f: G= \F_q^n\to[0,\infty)$ satisfy
\[
\sum_{x\in G} f(x)\ge \delta N
\qquad\text{and}\qquad
\sum_{x\in G} f(x)^s\le C N
\]
for some $\delta\in(0,1]$ and $C\ge 1$. Then
\[
\sum_{a_1x_1+\cdots+a_kx_k=0}\prod_{j=1}^k f(x_j)
\ \gg_{k,s,C,q}\
\delta^{\,k+\frac{s}{s-1}C_{q,k}}\;N^{k-1}.
\]
\end{lemma}

\begin{proof}
Let
\[
A_0:=\{x\in G:\ f(x)\ge \delta/2\}.
\]
By H\"older,
\[
\delta N
\le \sum_{x\in A_0}f(x)+\sum_{x\notin A_0}f(x)
\le |A_0|^{1-1/s}\Big(\sum_{x\in G}f(x)^s\Big)^{1/s}+\frac{\delta}{2}N,
\]
so $|A_0|\gg_{s,C}\delta^{s/(s-1)}N$. Writing $\rho:=|A_0|/N$, we have $\rho\gg_{s,C}\delta^{s/(s-1)}$.

Since $f\ge \delta/2$ on $A_0$,
\[
\sum_{a_1x_1+\cdots+a_kx_k=0}\prod_{j=1}^k f(x_j)
\ \ge\
\Big(\frac{\delta}{2}\Big)^k
\#\{(x_1,\dots,x_k)\in A_0^k:\ a_1x_1+\cdots+a_kx_k=0\}.
\]
Apply Corollary~\ref{cor:FF-supersat} to $A_0$ and use $\rho\gg_{s,C}\delta^{s/(s-1)}$ to obtain the stated bound.
\end{proof}

We are now ready to prove the main theorem of the section. 

\begin{proof}[Proof of Theorem~\ref{JPX2}]
Let $A\subseteq G$ be $K_{s,t}$-free and assume that every solution to \eqref{eq:FF-eqn} in $A^k$ is trivial.
Write
\[
|A|=\delta N^{1-1/s}.
\]

Let $0<\varepsilon<1$ be a parameter to be chosen at the end.
By Lemma~\ref{lem:dense-model-ff} (with $\eta=o(1)$ supplied by Lemma~\ref{lem: E_s}) there exists
$f:G\to[0,\infty)$ such that
\begin{enumerate}[label=(\roman*)]
\item $\sum_x f(x)=N^{1/s}|A|=\delta N$;
\item $\|\widehat f-N^{1/s}\widehat{\e_A}\|_\infty\ll \varepsilon N$;
\item $\sum_x f(x)^s\ll_{s,t} N$.
\end{enumerate}
Recall the definition
\[
T(h_1,\dots,h_k):=\sum_{a_1x_1+\cdots+a_kx_k=0}\prod_{j=1}^k h_j(x_j),
\qquad T(h):=T(h,\dots,h).
\]
Applying Lemma~\ref{lem:FF-Roth-count} to $f$ yields
\begin{equation}\label{eq:Tf-lower-FF}
T(f)\gg_{k,s,t,q}\delta^{\,k+\frac{s}{s-1}C_{q,k}}\,N^{k-1}.
\end{equation}
Set $\nu:=f+N^{1/s}\e_A\ge 0$ and $g:=f-N^{1/s}\e_A$. Then $\|\widehat g\|_\infty\ll \varepsilon N$ by (ii).
As in the integer case, using (iii) and Lemma~\ref{lem: E2} one has $\sum_x\nu(x)\ll N$ and $E_2(\nu,\nu)\ll N^3$.
Lemma~\ref{lem: counting} therefore applies in $G$ and, via the same telescoping identity as before, gives
\begin{equation}\label{eq:transfer-FF}
|T(f)-T(N^{1/s}\e_A)|\ll_k \varepsilon\,N^{k-1}.
\end{equation}
Since $A$ has only trivial solutions to \eqref{eq:FF-eqn},
\[
T(N^{1/s}\e_A)=N^{k/s}\cdot |A|=\delta\,N^{1+\frac{k-1}{s}}.
\]
Combining with \eqref{eq:transfer-FF} gives
\[
T(f)\ll N^{1+\frac{k-1}{s}}+\varepsilon N^{k-1}.
\]
Comparing with \eqref{eq:Tf-lower-FF} and dividing by $N^{k-1}$ yields, for $N$ large,
\[
\delta^{\,k+\frac{s}{s-1}C_{q,k}}\ll_{k,s,t,q}\varepsilon.
\]
Finally choose $\varepsilon:=(\log N)^{-c_0}$ with $c_0>0$ sufficiently small. Then
\[
\delta\ll_{k,s,t,q}(\log N)^{-c},
\qquad
c:=\frac{c_0}{\,k+\frac{s}{s-1}C_{q,k}\,}>0,
\]
and hence $|A|=\delta N^{1-1/s}\ll N^{1-1/s}(\log N)^{-c}$.
\end{proof}

\end{document}